\documentclass[11pt,english]{article}
\usepackage[latin9]{inputenc}
\usepackage[letterpaper]{geometry}
\usepackage{amsmath,amssymb,amsthm}
\usepackage{graphicx}
\usepackage{setspace}
\usepackage{babel}
\usepackage{float}

\theoremstyle{plain}
\newtheorem{theorem}{Theorem}
\newtheorem*{theorem*}{Theorem}

\newtheorem*{corollary*}{Corollary}
\newtheorem{definition}{Definition}
\newtheorem*{remark*}{Remark}

\def\mathscr{\mathfrak}

\newcommand{\ds}{\displaystyle}

\def\R{{\mathbb R}}

\def\Ker{\mathrm{Ker\,}}
\def\Ran{\mathrm{Ran\,}}

\def\l{\lambda}

\def\pfi{\varphi}
\def\D{\mathcal{D}}
\def\d{\partial}
\def\-{\backslash}

\def\ds{\displaystyle}

\makeatletter

\date{}

\makeatother

\begin{document}

\title{A short proof of the Arendt-Chernoff-Kato theorem}

\author{Sergiy Koshkin\\
 Department of Computer and Mathematical Sciences\\
 University of Houston-Downtown\\
 One Main Street\\
 Houston, TX 77002\\
 e-mail: koshkins@uhd.edu}
\maketitle
\begin{abstract}\

We give a short new proof of the Arendt-Chernoff-Kato theorem, which characterizes generators of positive $C_0$ semigroups in order unit spaces. The proof avoids half-norms and subdifferentials, and is based on a sufficient condition for an operator to have positive inverse, which is new even for matrices.
\bigskip

\end{abstract}

Let $A$ be a matrix with non-negative off-diagonal entries, such matrices are called positive off-diagonal, and consider the equation $Ax=-z$, where $z$ is a vector with strictly positive entries. A classical result states that if this equation has a positive solution $x$ then $-A^{-1}$ exists and has non-negative entries \cite[23.1]{Coll}. As shown below, the result generalizes to operators in order unit spaces. Moreover, it turns out that negative invertibility follows from positive solvability of $Ax=-z$ even if $A$ is not positive off-diagonal, but has a weaker property that we call somewhere positivity. Unlike positivity off-diagonal, somewhere positivity makes sense even for operators between different spaces. Based on this generalization we give a simple proof of the Arendt-Chernoff-Kato theorem without using half-norms or subdifferentials as in \cite{ACK}. Recall that the theorem characterizes generators of positive $C_0$ semigroups on order unit spaces as densely defined positive off-diagonal linear operators that satisfy a range condition. In addition to simplifying the proof we replace the positive off-diagonal property with an a priori weaker one, which is easier to verify.

Let $X$ be a real Banach space partially ordered by a closed proper cone $X^+$ with non-empty interior, 
$\mathrm{int\,}X^+$ and $\d X^+$ denote the interior and the boundary of $X^+$ respectively. If $X^+$ is also normal, i.e. order intervals are norm bounded, then $X$ is called an order unit space. In such spaces any element $e\in\mathrm{int\,}X^+$ is called an order unit and defines an equivalent norm $||x||_e:=\inf\{\l>0\,|-\l e\leq x\leq\l e\}$ on $X$ \cite[A.2.7]{Clem}. The dual cone $X^{*+}:=\{\pfi\in X^{*}|\,\langle\pfi,x\rangle\geq0\text{ for all }x\in X^+\}$ then is also closed, proper and normal in the dual norm. An element $x\in X^+$ is quasi-interior if $\langle\pfi,x\rangle>0$ for all $\pfi\in X^{*+}\-\{0\}$. By a theorem of Krein, in order unit spaces any quasi-interior element is an order unit \cite[A.2.10]{Clem}, but such elements may exist even if $\mathrm{int\,}X^+=\emptyset$. Let $X$ be an order unit space and $Y$ be an ordered Banach space.
\begin{definition}  A densely defined linear operator $A:\D_A\to Y$ with domain $\D_A\subseteq X$ is called somewhere positive if for every $x\in\D_A\cap\d X^+$ there is $\psi\in Y^{*+}\-\{0\}$ such that $\langle\psi,Ax\rangle\geq0$.
\end{definition}
If $X=\R^n$ and $Y=\R^m$ with the standard order a somewhere positive matrix $A$ must have a non-negative entry in every column, but this condition is not sufficient. One sufficient condition is that after deleting any column from $A$ the resulting matrix has a positive row. There does not appear to be a simple condition on matrix entries that is both necessary and sufficient. We now relate positive solvability of $Ax=-z$ to inverse positivity of 
$-A$ for somewhere positive $A$.
\begin{theorem}\label{qint} Let $A:\D_A\to Y$ be a somewhere positive operator and $z\in Y^{+}$ be a quasi-interior element. If the equation $Ax=-z$ has an order unit solution $x=e\in\mathrm{int\,}X^+$ then $A$ is inverse positive, i.e. 
$Ax\leq0$ implies $x\geq0$, and $\Ker A=\{0\}$. If additionally $Y^{+}$ is normal and $\Ran A=Y$, then $A$ is invertible and $-A^{-1}\geq0$.
\end{theorem}
\begin{proof}By assumption, there is $e\in\D_A\cap\mathrm{ int\,}X^+$ such that $Ae=-z$. Suppose by contradiction that there is some $x\in\D_A$ with $Ax\leq0$ but $x\not\in X^+$. Define a linear path $x_t:=(1-t)x+te$ connecting $x$ to $e$.
We have $x_0=x\not\in X^+$ and $x_1=e\in\mathrm{int\,}X^+$, so there is a borderline value $t_b\in(0,1)$ such that 
$x_{t_b}\in\d X^+$. Since $A$ is somewhere positive there is $\psi\in Y^{*+}\-\{0\}$ such that 
$\langle\psi,Ax_{t_b}\rangle\geq0$. On the other hand, $\langle\psi,z\rangle>0$ since $z$ is a quasi-interior element, and 
$Ax\leq0$, so
$$
\langle\psi,Ax_{t_b}\rangle=(1-t_b)\langle\psi,Ax\rangle+t_b\,\langle\psi,Ae\rangle
=(1-t_b)\langle\psi,Ax\rangle-t_b\,\langle\psi,z\rangle<0.
$$
This is the contradiction proving inverse positivity: $t_b$ can not exist and $x\in X^+$ .

Now let $x\in\Ker A$, then $Ax=A(-x)=0$. By inverse positivity we have both $x\geq0$ and $-x\geq0$, so $x=0$. 
If $\Ran A=Y$ then $A$ has an algebraic inverse on $Y$, and $-A^{-1}\geq0$ again by inverse positivity. However, any positive linear operator from a space with a generating cone to a space with a normal cone is bounded 
\cite[A.2.11]{Clem}, hence $A$ has a bounded inverse.
\end{proof}
\noindent Theorem \ref{qint} appears to be new even for matrices, where the stronger property of positive off-diagonal is usually required to derive inverse positivity, see \cite[23.1]{Coll}. The following definition is a well-known generalization of this property to infinite-dimensional spaces \cite[7.1]{Clem}.
\begin{definition} Let $X$ be an order unit space. A densely defined linear operator 
$A:\D_A\to X$ with domain $\D_A\subseteq X$ is positive off-diagonal if for every $x\in\D_A\cap\d X^+$ and  
every $\pfi\in X^{*+}\-\{0\}$:  $\langle\pfi,x\rangle=0\implies\langle\pfi,Ax\rangle\geq0$.
\end{definition}
\noindent If $x\in\d X^+$ then there always exists $\pfi\in X^{*+}\-\{0\}$ such that $\langle\pfi,x\rangle=0$, otherwise $x$ would be quasi-interior and then interior since $X$ is an order unit space. Therefore, any positive off-diagonal operator is somewhere positive. The converse is not true even for $X=Y=\R^n$ with the standard order. For $n\times n$ matrices positive off-diagonal literally means non-negative off-diagonal entries \cite[7.19]{Clem}, so 
$A=\begin{pmatrix}1&-2\\-2&1\end{pmatrix}$ is somewhere positive, but obviously not positive off-diagonal. Moreover, it also satisfies the solvability assumption of Theorem \ref{qint} with $z=(1,1)^T$, and therefore has negative inverse.

It is the positive off-diagonal property that appears in the Arendt-Chernoff-Kato characterization of generators of positive $C_0$ semigroups on order unit spaces \cite{ACK}, \cite[7.29]{Clem}. A key difference with somewhere positivity is that positive off-diagonal is preserved under subtraction of scalar operators, i.e. if $A$ is positive off-diagonal then so is $A-\l I$ for any $\l>0$. The same holds for the following ostensibly weaker property. 
\begin{definition} Let $X$ be an order unit space. A densely defined linear operator 
$A:\D_A\to X$ with domain $\D_A\subseteq X$ is somewhere positive off-diagonal if for every $x\in\D_A\cap\d X^+$ there exists $\pfi\in X^{*+}\-\{0\}$ such that $\langle\pfi,x\rangle=0$ and $\langle\pfi,Ax\rangle\geq0$. 
\end{definition}
\noindent Since $\langle\pfi,(A-\l I)x\rangle=\langle\pfi,Ax\rangle-\l\langle\pfi,x\rangle=\langle\pfi,Ax\rangle\geq0$
it follows that $A-\l I$ is somewhere positive off-diagonal along with $A$ for any $\l>0$\,. 

Our proof of the Arendt-Chernoff-Kato theorem is a simple application of Theorem \ref{qint}, which also weakens the positive off-diagonal assumption to somewhere positive off-diagonal. The key idea is that for large enough $\l>0$ the equation $(A-\l I)x=-z$ is solvable for some order unit $z$ with a solution $x$ itself also being an order unit. We will also make use of the fact that if $B$ is a bounded linear operator between order unit spaces with units $e$ and $\epsilon$ respectively, then $||B||=||Be||_{\epsilon}$, where $||\cdot||$ is the induced operator norm. The proof is straightforward from definitions and is left to the reader.
\begin{theorem}[Arendt-Chernoff-Kato]\label{ACK} Let $X$ be an order unit space and $A$ be a densely defined linear operator on $X$. Then the following conditions are equivalent:

{\rm(i)} $A$ generates a positive $C_0$ semigroups on $X$;

{\rm(ii)} $A$ is positive off-diagonal and $\Ran(A-\l I)=X$ for all large $\l>0$;

{\rm(iii)} $A$ is somewhere positive off-diagonal and $\Ran(A-\l I)=X$ for all large $\l>0$;

{\rm(iv)} $\l$ is in the resolvent set of $A$, and $(\l I-A)^{-1}\geq0$ for all large $\l>0$.
\end{theorem}
\begin{proof} ${\rm(i)}\Rightarrow{\rm(ii)}$ is straightforward, see e.g. \cite[7.18]{Clem}, and 
${\rm(ii)}\Rightarrow{\rm(iii)}$ is obvious from definitions.

${\rm(iii)}\Rightarrow{\rm(iv)}$ Let $e\in\D_A\cap\mathrm{int\,}X^+$, which is non-empty since $A$ is densely defined. Then $Ae\leq||Ae||_ee$ by definition of the order unit norm. Taking $\l>||Ae||_e$ we have 
$(A-\l I)e=:-z\leq-\varepsilon e$, where $\varepsilon:=\l-||Ae||_e>0$ and $z\geq\varepsilon e$ is itself an order unit.
Since somewhere positive off-diagonal property is preserved under subtraction of scalars $A-\l I$ satisfies all the assumptions of Theorem \ref{qint}. Therefore, $-(A-\l I)^{-1}=(\l I-A)^{-1}\geq0$.

${\rm(iv)}\Rightarrow{\rm(i)}$ Again let $e\in\D_A\cap\mathrm{int\,}X^+$, and assume for the moment that $Ae\leq0$. Since $(\l I-A)^{-1}=\l^{-1}(I-\l^{-1}A)^{-1}$ we conclude that $(I-\alpha A)^{-1}\geq0$ for small positive $\alpha$. Together with $Ae\leq0$ this implies $(I-\alpha A)\,e=e-\alpha Ae\geq e$ and $0\leq(I-\alpha A)^{-1}e\leq e$. Since the order unit norm is monotone 
$||(I-\alpha A)^{-1}e||_{e}\leq||e||_{e}=1$. Therefore, $||(I-\alpha A)^{-1}e||=||(I-\alpha A)^{-1}e||_{e}\leq1$, and by the Hille-Yosida theorem $A$ generates a contraction $C_0$ semigroup $T(t)$ on $X$ \cite[II.3.5]{EN}. Moreover, $\ds{T(t)=s\text{\,--\!}\lim_{n\to\infty}\Big(I-\frac{t}{n}A\Big)^{-n}\geq0}$. 

If $Ae\not\leq0$ we can argue as above with $A-||Ae||_{e}I$ in place of $A$. Indeed, 
$(A-||Ae||_{e}I)e\leq0$ by definition of the order unit norm, and 
$$
\Bigl(\lambda I-(A-||Ae||_{e}I)\Bigr)^{-1}=\Bigl((\lambda+||Ae||_{e})I-A\Bigr)^{-1}\geq0
$$
for large positive $\lambda$. Therefore, $A-||Ae||_{e}I$ generates a contraction $C_0$ semigroup on $X$. But then $A$ generates the $C_0$ semigroup $e^{t||Ae||_{e}}T(t) \geq0$ and we are done.
\end{proof}
\noindent The idea of proving ${\rm(iv)}\Rightarrow{\rm(i)}$ is similar to Arendt's \cite{Ar}, but our approach is more elementary. The most novel part of the proof is ${\rm(iii)}\Rightarrow{\rm(iv)}$, where we replaced the use of half-norms and subdifferentials with an application of our Theorem \ref{qint}. 

It follows from Theorem \ref{ACK} that the somewhere positive off-diagonal property together with the range condition yield the usual positive off-diagonal property. Since for $n\times n$ matrices the range condition follows from $\Ker A=\{0\}$ Theorem \ref{qint} implies that for them the two properties are equivalent. This is in contrast with somewhere positivity, which is strictly weaker even for matrices. Note however, that somewhere positive off-diagonal property is easier to verify in examples since one does not have to deal with arbitrary elements of $X^{*+}$.

\end{document}